\documentclass{amsart}
\usepackage[left=2cm, right=2cm]{geometry}
\usepackage[utf8]{inputenc}
\usepackage{amsfonts, amsthm, amssymb, mathtools,etoolbox,stmaryrd}
\usepackage{blindtext}
\usepackage[colorlinks=true, urlcolor=blue, linkcolor=blue, citecolor=blue]{hyperref}
\usepackage{tikz,tikz-cd}
\usetikzlibrary{calc}
\usepackage{cleveref}
\usepackage{array}
\usepackage[style=alphabetic,sorting=nyt]{biblatex}
\renewbibmacro{in:}{}
\tikzset{pullback/.style={minimum size=1.2ex,path picture={
\draw[opacity=1,black,-,#1] (-0.5ex,-0.5ex) -- (0.5ex,-0.5ex) -- (0.5ex,0.5ex);%
}}}
\usetikzlibrary{positioning}

\theoremstyle{plain}
\newtheorem{theorem}{Theorem}[section]
\newtheorem{proposition}[theorem]{Proposition}
\newtheorem{lemma}[theorem]{Lemma}
\newtheorem{corollary}[theorem]{Corollary}

\theoremstyle{definition}
\newtheorem{example}[theorem]{Example}
\newtheorem{definition}[theorem]{Definition}
\newtheorem{remark}[theorem]{Remark}

\newcommand{\dq}[1]{``#1"}

\newcommand{\invmemo}[1]{}

\newcommand{\N}{\mathbb{N}}
\newcommand{\Z}{\mathbb{Z}}

\newcommand{\R}{\mathbb{R}}
\newcommand{\C}{\mathcal{C}}
\newcommand{\vC}{\overline{\C}}

\newcommand{\E}{\mathcal{E}}
\newcommand{\F}{\mathcal{F}}
\newcommand{\G}{\gamma}
\newcommand{\id}{\mathrm{id}}
\newcommand{\op}{\mathrm{op}}
\newcommand{\ob}{\mathrm{ob}}
\newcommand{\Set}{\mathbf{Set}}
\newcommand{\FinSet}{\mathbf{FinSet}}
\newcommand{\PSh}{\mathbf{PSh}}
\newcommand{\Sh}{\mathbf{Sh}}
\newcommand{\sSet}{\mathbf{sSet}}
\newcommand{\Func}[2]{[#1,#2]}

\newcommand{\demph}[1]{\textbf{#1}}
\newcommand{\Spec}{\mathrm{Spec}}
\newcommand{\vo}{{\boxempty}}

\newcommand{\Ax}[1]{\mathbf{adj}_{#1}}
\newcommand{\conn}{\mathrm{conn}}
\newcommand{\Sub}{\mathrm{Sub}}
\newcommand{\Fam}{\mathbf{Fam}}
\newcommand{\Trees}{\mathbf{Trees}}
\newcommand{\Forests}{\mathbf{Forests}}
\DeclareMathOperator*{\colim}{colim}
\renewcommand{\O}{\mathcal{O}}
\newcommand{\tal}{\triangleleft}
\newcommand{\Gl}{\mathbf{Gl}}
\newcommand{\T}{\mathbb{T}}

\title[Completely connected topos]{Grothendieck topoi with a left adjoint to a left adjoint to a left adjoint to the global sections functor}
\author{Ryuya Hora}
\thanks{Graduate School of Mathematical Sciences, University of Tokyo. \url{hora@ms.u-tokyo}}

\begin{document}

\begin{abstract}
This paper introduces the notion of complete connectedness of a Grothendieck topos, defined as the existence of a left adjoint to a left adjoint to a left adjoint to the global sections functor, and provides many examples.
Typical examples include presheaf topoi over a category with an initial object, such as the topos of sets, the Sierpi\'nski topos, the topos of trees, the object classifier, the topos of augmented simplicial sets, and the classifying topoi of many algebraic theories, such as groups, rings, and vector spaces.

	We first develop a general theory on the length of adjunctions between a Grothendieck topos and the topos of sets.
We provide a site characterisation of complete connectedness, which turns out to be dual to that of local topoi. We also prove that every Grothendieck topos is a closed subtopos of a completely connected Grothendieck topos.
\end{abstract}
\maketitle

\tableofcontents
\section{Introduction}

Some properties of Grothendieck topoi (over the base topos of sets) are defined by the existence of adjoint functors to the global sections functor. For example, a Grothendieck topos $\E$ is said to be \demph{locally connected} if a left adjoint to a left adjoint to the global sections functor exists.
For a topological space $X$, its sheaf topos $\Sh(X)$ is locally connected in this sense if and only if the original topological space $X$ is locally connected. For another example, a Grothendieck topos $\E$ is said to be \demph{local} if the global sections functor has a right adjoint. For a commutative ring $A$, its sheaf topos $\Sh(\Spec A)$ is local in this sense if and only if the ring $A$ is a local ring.
Another characterisation theorem of this kind is in 
\cite{rosebrugh1994adjoint},
which proves that a category $\C$ is equivalent to $\Set$ if and only if it admits a left adjoint to a left adjoint to a left adjoint to a left adjoint to the Yoneda embedding. 

In this paper, we first describe all possible lengths of adjoint strings between a Grothendieck topos and the topos of sets. Starting from the unique left exact cocontinuous functor $\G_0 \colon \Set \to \E$, which is usually denoted by $\G^*$ or $\Delta$, we consider the existence of left adjoint sequence $\G_n \dashv \G_{n-1} \dashv \dots \dashv \G_0$ and the existence of right adjoint sequence $\G_0 \dashv \G_{-1} \dashv \dots \dashv \G_{-m}$. In \Cref{sec:AdjointStrings}, we will observe that these conditions define five different classes of Grothendieck topoi, including the class of all Grothendieck topoi, the class of locally connected Grothendieck topoi, the class of local Grothendieck topoi, and two others (\Cref{thm:AxiomsImplication}).

In \Cref{sec:CompletelyConnectedTopoi}, we will focus on one of the five classes of Grothendieck topoi, which we call \demph{completely connected topoi}. More concretely, a Grothendieck topos is completely connected if it has a left adjoint to a left adjoint to a left adjoint to the global sections functor $\G_2 \dashv \G_{1} \dashv \G_{0} \dashv \G_{-1} \colon \E \to \Set$. 
Our terminology is based on the fact that complete connectedness implies other notions of connectedness, including
connectedness, 
local connectedness,
stable local connectedness \cite{johnstone2011remarks}, and total connectedness \cite{bunge1996spreads} (\Cref{rmk:ComparisonWithOtherConnectedness}).
Unsurprisingly, completely connected topoi have many special properties. 
For example, in a completely connected topos, connected objects are closed under small limits (\Cref{prop:OtherPropertiesOFCCtopos}).
 Although complete conectedness is a very strong condition, it is not too restrictive in the sense that every Grothendieck topos is a closed subtopos of a completely connected topos (\Cref{cor:toposIsClosedSubtoposOfCompletelyConnectedTopos}). We will provide many examples in \Cref{sec:examples}.

What might be surprising is that completely connected topoi are dually similar to local topoi. 
For example, it turns out to be easy to show
that $\PSh(\C)$ is completely connected if and only if $\PSh(\C^\op)$ is local (\Cref{cor:PresheafDuality}).
With a more detailed discussion, we provide a site characterisation of completely connected topoi, which is dual to the site characterisation of local topoi (\Cref{thm:SiteCharacterisation}). Furthermore, the Freyd cover construction of topoi, which makes a topos into a local topos, has a dually similar construction $\E \mapsto \Fam(\E)$, which makes a topos into a completely connected topos
(\Cref{thm:FamIsCompletelyConnected} and \Cref{rmk:FamilyAsArtinGluing}).
While the terminal object plays a central role in the characterisation theorems of local topoi, we define the notion of \demph{the container object} (\Cref{def:containerObject}), which plays the corresponding role in our characterisations of completely connected topoi.

\subsection*{Acknowledgement}
The author would like to thank his supervisor, Ryu Hasegawa, for his continuous support. He is also grateful to Matias Menni and Axel Osmond for their insightful and encouraging advice, and to Yuki Maehara for his careful reading and for pointing out typos and errors.
He also thanks Yuto Kawase, Satoshi Nakata, and Yuno Suzuki for the valuable discussions and comments.
He was supported by JSPS KAKENHI Grant Number JP24KJ0837 and FoPM, WINGS Program, the University of Tokyo.

\section{Adjoint strings between a Grothendieck topos and the base topos of sets}\label{sec:AdjointStrings}
It is well known that several important properties of a Grothendieck topos $\E$ are described in terms of the adjunction between $\E$ and the (base) topos $\Set$.
We start with a related categorical puzzle: How long can maximal adjoint strings between $\E$ and $\Set$ be? (This will be answered in \Cref{rmk:PossibleLengthOfMaximalAdjoint}.)


\begin{definition}
    For a Grothendieck topos $\E$, we define the functor $\G_n$ recursively as follows:
    \begin{itemize}
        \item $\G_0$ is the unique cocontinuous functor $\Set \to \E$ that preserves finite limits.
        \item $\G_{i+1}$ is the left adjoint to the functor $\G_{i}$, if it exists.
        \item $\G_{i-1}$ is the right adjoint to the functor $\G_{i}$, if it exists.
    \end{itemize}
    We say that a Grothendieck topos $\E$ satisfies $\Ax{i}$ if
    $\G_{i}$ exists.
    The condition that every $\G_i$ exists is referred to as $\Ax{\infty}$
\end{definition}
Notice that the functor $\G_{2k}$ is a functor from $\Set$ to $\E$, and $\G_{2k+1}$ is a functor from $\E$ to $\Set$ whenever it exists.

\begin{remark}[Meaning of each $\G_i$]\label{rmk:MeaningOfGamma}
The functor $\G_i$ for $i=0,-1, 1$ has the following conventional interpretations:
\begin{description}
    \item[$\G_{0}$] The functor $\G_0\colon \Set \to \E$, which is usually called \demph{the constant sheaf functor}, sends a set $S$ to the coproduct of $S$-indexed copies of the terminal sheaf $1_{\E}\in \ob(\E)$. This functor is usually denoted by $\Delta$ or $\G^*$.
    \item[$\G_{-1}$] The functor $\G_{-1}\colon \E \to \Set$, which is usually called \demph{the global sections functor}, sends a sheaf $X\in \ob(\E)$ to the set of its global sections $\E(1_{\E}, X)$. This functor is usually denoted by $\Gamma$ or $\G_*$.
    \item[$\G_{1}$] The functor $\G_1 \colon \E \to \Set$, which is usually called \demph{the connected components functor}, sends a sheaf $X$ to the set of its connected components if such a functor $\G_1$ exists. This functor is usually denoted by $\Pi_0$ or $\G_!$.
\end{description}
\end{remark}
Notice that the conditions $\Ax{0}$ and $\Ax{-1}$ always hold for all Grothendieck topoi, since the adjunction $\G_{0}\dashv \G_{-1}$ is the global sections geometric morphism $\E \to \Set$. 
The condition $\Ax{1}$ is equivalent to being locally connected, and $\Ax{-2}$ is equivalent to being local \footnote{These hold since we consider $\Set$ as the base topos. This paper does not study the relative case, which may be of interest.}. 


The following examples will be needed to prove \Cref{thm:AxiomsImplication}.

\begin{example}[The topos of idempotents]\label{exmpl:ToposOfIdempotents}
    Consider the topos of idempotents, whose object is a pair $(X,e)$ of a set $X$ and an idempotent endofunction $e\colon X\to X$, and whose morphism is a function compatible with the associated idempotents. This is equivalent to the presheaf category over the monoid $(\mathbb{F}_2,1, \times)$. For this topos, the $\G_1$ functor is naturally isomorphic to the global sections functor $\G_{-1}$, since each connected component of an object $(X,e)$ has exactly one fixed point. Thus, the topos admits the infinite adjoint string $\cdots \dashv \G_{-1} \dashv \G_0 \dashv \G_{-1} \dashv \cdots$ and satisfies $\Ax{\infty}$.
\end{example}

\begin{example}[The Sierpi\'nski topos]\label{exmpl:ToposOfFunctions}
    The topos of functions $\Set^{\to}$, which is also known as the Sierpi\'nski topos, admits the following adjoint $5$-tuple:
    \[
    \begin{tikzcd}[column sep =90pt]
        \Set^{\to}\ar[r, shift left = 24pt,"\G_{1}"] \ar[r, shift right = 8pt,"\G_{-1}"] &\Set,\ar[l,"\G_2"', shift right = 40pt]\ar[l,"\G_0"', shift right = 8pt]\ar[l,"\G_{-2}"', shift left = 24pt]
    \end{tikzcd}
    \]
    where the functor $\G_{2}$ sends a set $X$ to the function $\emptyset \to X$, and the functor $\G_{-2}$ sends a set $X$ to the function $X \to {\ast}$. Since $\G_2$ does not preserve the terminal object and $\G_{-2}$ does not preserve the initial object, this adjoint $5$-tuple cannot be extended to an adjoint $6$-tuple. This shows that the topos $\Set^{\to}$ satisfies $\Ax{2}$ and $\Ax{-2}$ but does not satisfy $\Ax{\infty}$.
\end{example}

\begin{example}[The sheaf topos over a circle]\label{exmpl:SheafToposOverCircle}
    Let $S^1$ denote the circle regarded as a topological space. The unique geometric morphism $\Sh(S^1)\to \Set$ is locally connected, defining the adjoint triple $\G_1 \dashv \G_0 \dashv \G_{-1}$. 
    \[
    \begin{tikzcd}[column sep = 40pt]
        \Sh(S^1)\ar[r,shift left =10pt,"\G_1"]\ar[r,shift right=10pt, "\G_{-1}"]&\Set\ar[l,"\G_0"']
    \end{tikzcd}
    \]
    This cannot be extended to an adjoint $4$-tuple since $\G_{-1}$ does not preserve coequalizers and $\G_{1}$ does not preserve binary products. To see the latter, one can consider the universal covering $p\colon \R \to S^1$. Although it is connected $\G_1(\R)=1$, its square $\R\times_{S^1} \R \cong \coprod_{i\in \Z}\R \to S^1$ is a countable coproduct of the universal covering. This shows that $\Sh(S^1)$ satisfies $\Ax{1}$, but satisfies neither $\Ax{2}$ nor $\Ax{-2}$.
\end{example}

\begin{example}[The sheaf topos over the Cantor space]\label{exmpl:SheafToposOverCantorSpace}
    For the Cantor space $2^\N$, the unique geometric morphism $\Sh(2^{\N})\to \Set$ cannot be extended to an adjoint triple. The nonexistence of $\G_1$ follows since $2^{\N}$ is not locally connected. Thus, the topos $\Sh(2^\N)$ satisfies neither $\Ax{1}$ nor $\Ax{-2}$.
\end{example}

\begin{theorem}\label{thm:AxiomsImplication}
We have the following implications between the following conditions of a Grothendieck topos $\E$.
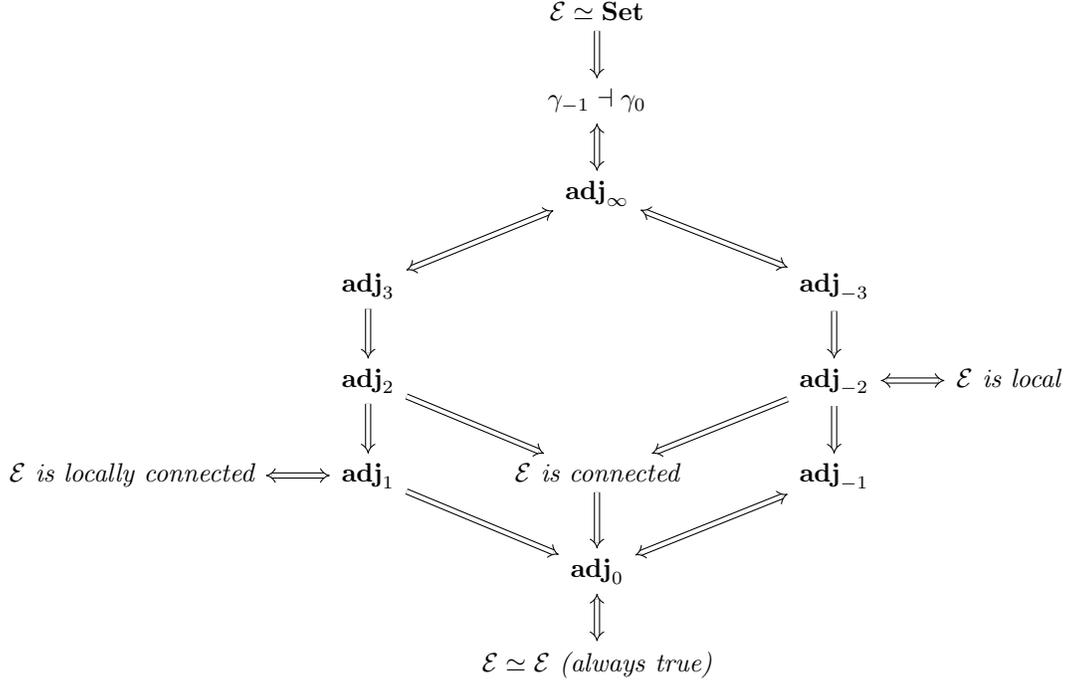
\[
\begin{tikzcd}
    &&\E\simeq \Set \ar[d, Rightarrow]&&\\
    &&\G_{-1} \dashv \G_{0} \ar[d, Leftrightarrow]&&\\
    &&\Ax{\infty}\ar[rd, Leftrightarrow]\ar[ld, Leftrightarrow]&&\\
     &\Ax{3}\ar[d, Rightarrow]&&\Ax{-3}\ar[d, Rightarrow]&\\
    &\Ax{2}\ar[d, Rightarrow]\ar[rd, Rightarrow]&&\Ax{-2}\ar[d, Rightarrow]\ar[r,Leftrightarrow]\ar[dl, Rightarrow]&\E\text{ is local}\\
     \E\text{ is locally connected}&\Ax{1}\ar[rd, Rightarrow]\ar[l, Leftrightarrow]&\E\text{ is connected}\ar[d, Rightarrow]&\Ax{-1}\ar[ld, Leftrightarrow]&\\
     &&\Ax{0}\ar[d,Leftrightarrow]&&\\
     &&\E\simeq \E\text{ (always true)}&&
\end{tikzcd}
\]
Furthermore, for each implication symbol $\implies$, the converse implication does not hold.
\end{theorem}

\begin{proof}
We will prove them one by one.
\begin{description}
    \item[$\E\simeq \Set \implies \G_{-1} \dashv \G_{0}$] This holds trivially, since $\G_{i} \cong \id_{\Set}$ for any $i\in \Z$. \Cref{exmpl:ToposOfIdempotents} shows that the converse does not hold.
    \item[$\G_{-1} \dashv \G_{0} \iff \Ax{\infty} \iff \Ax{3} \iff \Ax{-3}$] It is easy to prove $\G_{-1} \dashv \G_{0} \implies \Ax{\infty} \implies \Ax{3}$ and $\Ax{\infty} \implies \Ax{-3}$. Conversely, the condition $\Ax{3}$ implies that $\G_2$ is continuous and cocontinuous and hence $\G_2\cong \G_0$. This implies $\G_{-1} \cong \G_1 \dashv \G_{0}$. Similarly, the condition $\Ax{-3}$ implies that $\G_{0} \cong \G_{-2}$ and $\G_{-1} \dashv \G_{-2} \cong \G_{0}$.
    \item[Middle cases] The implications $\Ax{3}\implies \Ax{2}$, $\Ax{2}\implies \Ax{1}$, $\Ax{1}\implies \Ax{0}$, $\Ax{-3}\implies \Ax{-2}$, and $\Ax{-2}\implies \Ax{-1}$ are trivial by definition. It is shown that the converse do not hold by \Cref{exmpl:ToposOfFunctions}, \Cref{exmpl:SheafToposOverCircle}, \Cref{exmpl:SheafToposOverCantorSpace}, \Cref{exmpl:ToposOfFunctions}, and \Cref{exmpl:SheafToposOverCantorSpace} respectively.    
    \item[$\E\simeq \E \iff \Ax{-1}\iff \Ax{0}$] This follows since every Grothendieck topos admits the geometric morphism $\G_0\dashv \G_{-1}\colon \E\to \Set$.
    \item[$\Ax{2} \implies \E \text{ is connected}$] Assuming $\Ax{2}$, the functor $\G_1$ preserves the terminal object $\G_1(1_\E) \cong 1_{\Set}$ since $\G_1$ has the left adjoint $\G_{2}$. This shows that the terminal object $1_{\E}$ is connected and the topos $\E$ is connected.
    The converse does not hold, as exemplified by \Cref{exmpl:SheafToposOverCircle}.
    \item[$\Ax{-2} \implies \E \text{ is connected}$] Assuming $\Ax{-2}$ the functor $\G_{-1}\cong \E(1_{\E}, {-})$ preserves the binary coproduct $1_{\E}+1_{\E}$. This means that the terminal object $1_{\E}$ admits no nontrivial coproduct decomposition, which means that the topos $\E$ is connected. The converse does not hold, as exemplified by \Cref{exmpl:SheafToposOverCircle}.
     \item[$\E \text{ is connected} \implies \Ax{0}$] The implication is trivial by definition. The converse does not hold, as exemplified by \Cref{exmpl:SheafToposOverCantorSpace}.
\end{description}
\end{proof}

\begin{remark}[Possible lengths of maximal adjoint strings]\label{rmk:PossibleLengthOfMaximalAdjoint}
With \Cref{thm:AxiomsImplication}, we can conclude that the possible lengths of maximal adjoint $n$-tuples between $\Set$ and a Grothendieck topos $\E$ are $n=2,3,4,5, \infty$. In fact, if there is an adjoint $n$-tuple for $n\geq 4$, at least one of the $n$ functors coincides with the functor $\G_0$. An example of a maximal adjoint quadruple is given by the topos of simplicial sets $\sSet$.
\end{remark}

\begin{remark}[Relationship with quality types]
    Due to \Cref{thm:AxiomsImplication}, a Grothendieck topos $\E$ satisfies $\Ax{\infty}$ if and only if $\G_0 \colon \Set \to \E$ makes $\E$ a quality type over $\Set$ in the sense of \cite[][Definition 1]{lawvere2007axiomatic}.
\end{remark}

\section{Complete connectedness of a Grothendieck topos}\label{sec:CompletelyConnectedTopoi}

In this section, we will define the notion of completely connected Grothendieck topoi (\Cref{def:CompletelyConnectedTopoi}) and provide a site characterisation of them (\Cref{thm:SiteCharacterisation}).

\subsection{Definition}
In this subsection, we define the notion of complete connectedness by $\Ax{2}$. We will see many examples in \Cref{sec:examples}. Here, we only mention that the Sierpi\'nski topos (\Cref{exmpl:ToposOfFunctions}) is the prototypical example of a completely connected topos.
\begin{definition}[Complete connectedness]\label{def:CompletelyConnectedTopoi}
    We say that a Grothendieck topos $\E$ is \demph{completely connected} if $\E$ satisfies $\Ax{2}$.
\end{definition}

We can rephrase the complete connectedness of a Grothendieck topos in terms of the existence of a special object, which we will call a container object.
\begin{definition}[Container object]\label{def:containerObject}
    For a Grothendieck topos $\E$, an object $X\in \E$ is called 
    a \demph{container object} 
    if $\E(X, -)\colon \E \to \Set$ is a left adjoint to the funcor $\G_0\colon \Set \to \E$.
    The container object will be denoted by $\vo$ if it exists.
\end{definition}

By definition, a container object is unique up to isomorphisms if it exists. The reason for the notation $\vo$ will be explained in \Cref{exmp:SierpinskitoposAsCompletelyCOnnectedTopoi}.

\begin{proposition}[Immediate paraphrases]\label{prop:ImmediateRephrase}
    For a Grothendieck topos $\E$, the following conditions are equivalent.
    \begin{enumerate}
        \item $\E$ is completely connected.
        \item $\E$ has a container object.
        \item $\E$ is locally connected and $\G_1$ is representable.
        \item $\E$ is locally connected and $\G_1$ is continuous.
    \end{enumerate}
\end{proposition}
\begin{proof}
    The implication $(1)\implies (2)$ follows since $\G_2(1)$ is a contianer object:
    \[
    \E(\G_2(1), -) \cong \Set(1, \G_1(-)) \cong \G_1 .
    \]
    The equivalence $(2)\iff (3)$ immediately follows from the definition of a contianer object.
    The implication $(3) \implies (4)$ is also obvious since all representable functors are continuous.
    Lastly, assuming $(4)$, the easiest case of the adjoint functor theorem implies the existence of $\G_2$.
\end{proof}

\begin{remark}[Relation to other connectedness]\label{rmk:ComparisonWithOtherConnectedness}
    A Grothendieck topos $\E$ (relative to the base topos $\Set$) is called
\begin{itemize}
    \item \demph{locally connected} if the constant sheaf functor $\G_0\colon \Set \to \E$ has a left adjoint $\G_{1}$,
    \item \demph{stably locally connected} \cite{johnstone2011remarks} if it is locally connected and the functor $\G_{1}$ preserves finite products\footnote{This condition also appears in \cite[][]{lawvere1986categories} as \dq{Axiom $1$.}}, and
    \item \demph{totally connected} \cite{bunge1996spreads} if it is locally connected and the functor $\G_{1}$ preserves finite limits.
\end{itemize}
\Cref{prop:ImmediateRephrase} shows that the complete connectedness implies the all connectedness described above.
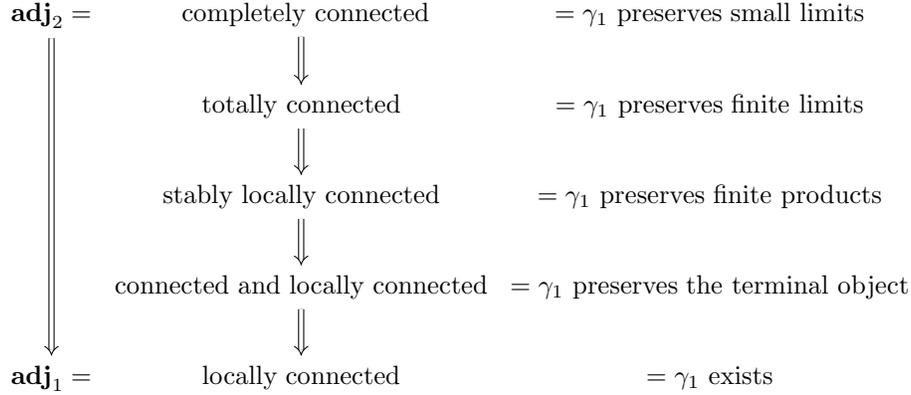
\[
\begin{tikzcd}[column sep = 0pt]
    \Ax{2}=\ar[dddd, Rightarrow]&\text{ completely connected}\ar[d, Rightarrow]&=\G_{1} \text{ preserves small limits}\\
    &\text{totally connected}\ar[d, Rightarrow]&=\G_{1} \text{ preserves finite limits}\\
    &\text{stably locally connected}\ar[d, Rightarrow]&=\G_{1} \text{ preserves finite products}\\
    &\text{connected and locally connected}\ar[d,Rightarrow]&=\G_{1} \text{ preserves the terminal object}\\
   \Ax{1}=& \text{locally connected}&=\G_{1} \text{ exists}\\
\end{tikzcd}
\]
\end{remark}

\subsection{Container object}
In this subsection, we observe basic properties of a container object.
Informally, a container object looks like \dq{an empty box}, which itself is not \dq{nothing}
(see \Cref{fig:ContainerObjectOfSierpinskitopos} and \Cref{fig:InitialObjectOfSierpinskitopos}).
The following proposition will be repeatedly used in the rest of the paper.
\begin{proposition}[Properties of the container object]\label{prop:FundamentalPropertiesOfContainerObject}
Let $\E$ be a completely connected topos and let $\vo$ be its container object. 
\begin{enumerate}
    \item $\vo$ is connected.
    \item $\vo$ is projective.
    \item $\vo$ is not initial.
    \item For any connected object $X\in \ob(\E)$, there is a unique morphism $\vo \to X$.
    \item $\vo$ is rigid, that is, the identity map $\id_{\vo}$ is the unique endomorphism of $\vo$.
    \item For any object $X\in \ob(\E)$, $X$ is not initial if and only if there exists at least one morphism $\vo \to X$.
\end{enumerate}
\end{proposition}
\begin{proof}
    First, we prove $(1),(2),$ and $(3)$.
    To prove the connectedness (respectively, projectivity) of $\vo$, it suffices to prove that $\E(\vo,-)$ preserves binary coproducts (respectively, epimorphisms). This follows sice $\E(\vo, -) \cong \G_1$ admits a right adjoint $\G_0$.
    The connectedness of $\vo$ implies that $\vo$ is not initial. 

    Next, we prove $(4), (5),$ and $(6)$. The functor $\G_1$ sends an object $X$ to the set of connected components of $X$ (\Cref{rmk:MeaningOfGamma}). Since $\vo$ represents the functor $\G_1$, there is a unique morphism $\vo \to X$ for each connected object $X\in \ob(\E)$.
    In particular, the connectedness of $\vo$ implies that $\vo$ is rigid. The last property $(6)$ follows from the fact that, if an object $X$ does not have connected conponents, then it is initial in a locally connected topos.
\end{proof}

The next proposition shows some unusual properties of completely connected topoi.
As we will not use it in the rest of the paper, the reader can skip this proposition. The full subcategory of a topos $\E$ consisting of connected objects will be denoted by $\E_{\conn}$.

\begin{proposition}[Properties of completely connected topoi]\label{prop:OtherPropertiesOFCCtopos}
    Every completely connected Grothendieck topos $\E$ satisfies the following conditions.
    \begin{enumerate}
        \item The canonical embedding $\E_{\conn} \hookrightarrow \E$ has a left adjoint.
        \item Small limits of connected objects are connected.
        \item The container object $\vo$ is an atom, i.e. $\vo$ has exactly two subobjects.
        \item The subobject classifier $\Omega$ has exactly two connected components.
    \end{enumerate}
\end{proposition}
\begin{proof}
    $(1):$ For each object $X\in \ob(\E)$, we define a morphism $\eta_X \colon X \to \overline{X}$ 
    by the following pushout diagram
    \[
    \begin{tikzcd}
        \G_2 \G_1 (X) \ar[d, "\G_2(!)"]\ar[r, "\epsilon_X"]& X \ar[d, "\eta_X"] \\
        \vo\ar[r]&\overline{X}.
    \end{tikzcd}
    \]
    Notice that $\G_2 (1_{\Set}) \cong \vo$.
    Since the left adjoint functor $\G_1$ preserves the pushout diagram, $\overline{X}$ is connected. For any connected object $Y\in \ob (\E)$, the universal property of the pushout shows that $\eta_X^* \colon \E(\overline{X},Y) \cong \E(X,Y)$ is bijective. Therefore, the map $\eta_X \colon X \to \overline{X}$ provides a reflector of the embedding $\E_{\conn} \hookrightarrow \E$.

    $(2):$  This follows from $(1)$.

    $(3):$ We have already shown that the container object $\vo$ is not initial (\Cref{prop:FundamentalPropertiesOfContainerObject} (3)).
    Take an arbitrary non-initial subobject $s\rightarrowtail \vo$. Due to \Cref{prop:FundamentalPropertiesOfContainerObject} (6), there is a morphism $\vo \to s$. Then, the diagram
    \[
    \begin{tikzcd}
        \vo\ar[r]\ar[rd,"\id_{\vo}"']&s\ar[d,tail]\\
        &\vo
    \end{tikzcd}
    \]
    commutes since $\vo$ is rigid (\Cref{prop:FundamentalPropertiesOfContainerObject} (5)). This proves that the monomorphism $s\rightarrowtail \vo$ is a split epimorphism, hence an isomorphism.

    $(4):$ This follows from $(3)$ since we have $\Sub(\vo) \cong \E(\vo, \Omega) \cong \G_1(\Omega)$.
\end{proof}

\subsection{Site characterisation}
This subsection aims to provide a site characterisation of completely conencted topoi.
We start by recalling the notion of irreducibility of an object in a site. See {\cite[][Definition C.2.2.18.(a).]{johnstone2002sketchesv2}} for more details.
\begin{definition}[Irreducible object]
    For a site $(\C,J)$, an object $c\in \C$ is said to be \demph{$J$-irreducible} if the only $J$-covering sieve on $c$ is the maximal sieve.
\end{definition}

\begin{remark}[Site characterisation of local topoi {\cite[][Example C.3.6.3 (d) local site]{johnstone2002sketchesv2}}]
    The site characterisation of local topoi is well known. A site $(\C,J)$ is called \demph{local} if $\C$ has a $J$-irreducible terminal object. A Grothendieck topos $\E$ is local if and only if $\E$ is (equivalent to) the sheaf topos over a small local site.
\end{remark}

Our site characterisation of completely connected topoi is dual to that of local topoi.
\begin{definition}
    We say that a site $(\C,J)$ is \demph{completely connected} if $\C$ has a $J$-irreducible initial object.
\end{definition}

\begin{theorem}\label{thm:SiteCharacterisation}
    A Grothendieck topos $\E$ is completely connected if and only if it is a sheaf topos over a small completely connected site.
\end{theorem}
\begin{proof}
    \textbf{If part:}
    Let $(\C,J)$ be a small completely connected site with a $J$-irreducible initial object $I$. First we will prove that $(\C,J)$ is a \textit{locally connected site}, i.e., a site in which every $J$-covering sieve $S$ on every object $c\in \ob(\C)$ is connected as a subcategory of $\C/c$ (see \cite[][C3.3.]{johnstone2002sketchesv2}).

    Take an arbitrary object $c\in \C$ and a $J$-covering sieve $S\in J(c)$. Since $J$ is closed under pullback of sieves, we have $(I \xrightarrow{!} c)^* S\in J(I) = \{\text{the maximal sieve on }I\}$, which implies $I \xrightarrow{!} c \in S$. Thus, the subcategory $S\subset \C/c$ is connected, since $I \xrightarrow{!} c \in S$ is an initial object of $\C/c$.

    The local connectedness of the site $(\C,J)$ implies that the constant presheaf fucntor $\Delta\colon  \Set \to \PSh(\C)$ lifts along the embedding $\Sh(\C,J) \hookrightarrow \PSh(\C)$
    \[
    \begin{tikzcd}
        &\Sh(\C,J)\ar[d,hook]\\
        \Set\ar[r,"\Delta"]\ar[ru, dashed, "\G_0"]&\PSh(\C),
    \end{tikzcd}
    \]
    (because the gluing condition for a constant presheaf becomes trivial by the adjunction $\colim_{\C} \dashv \Delta$). By construction, the lift $\Set \to \Sh(\C,J)$ coincides with the functor $\G_0$.

    Due to \Cref{prop:ImmediateRephrase}, it remains to prove that $\G_0$ has a representable left adjoint $\G_1$.
    By the above lifting diagram, we obtain the left adjoint
    \[
    \begin{tikzcd}
        \G_1 \colon \Sh(\C,J)\ar[r, hookrightarrow] & \PSh(\C)\ar[r, "\colim_{\C}"] &\Set.
    \end{tikzcd}
    \]
    Since the fucntor $\colim_{\C} \colon \PSh(\C) \to \Set$ is represented by $y(I) \in \ob(\PSh(\C))$, the sheafification of $y(I)$ represents the functor $\G_1$. This shows that $\Sh(\C,J)$ is a completely connected topos.

    \textbf{Only if part:} Take a generating set $G \subset \ob(\E)$ of the given Grothendieck topos $\E$. Since $\E$ is locally connected and well-powered, we may assume that $G$ consists of connected objects (by replacing them with their summands). Furthermore, we may assume that $G$ contains the container object $\vo$, since it is connected (\Cref{prop:FundamentalPropertiesOfContainerObject} (1)). Consider the small full subcategory $\C_G$ of $\E$ consisting of the objects in $G$ and the canonical topology $J_G$ on $\C_G$. 
    
    It suffices to prove that $(\C_G, J_G)$ is a compltely connected site. Since all objects in $\C_G$ are connected in $\E$, the container object $\vo$ is initial in $\C_G$ (\Cref{prop:FundamentalPropertiesOfContainerObject} (4)). Take an arbitrary $J_G$-covering ($=$ jointly epimorphic) sieve $S$ of $\vo$. Since $\vo$ is not initial in $\E$ (\Cref{prop:FundamentalPropertiesOfContainerObject} (3)), the sieve $S$ is nonempty. Take an element $f\colon c \to \vo$ in $S$.
    Since $c$ is a connected object in $\E$, there is a unique morphism $\vo \overset{s}{\to} c$ (\Cref{prop:FundamentalPropertiesOfContainerObject} (4)), so $\vo \overset{s}{\to} c \overset{f}{\to} \vo$ belongs to $S$. The rigidity of $\vo$ (\Cref{prop:FundamentalPropertiesOfContainerObject} (5)) implies $\id_{\vo}=f\circ s \in S$ and therefore $S$ is the maximal covering. This proves that the small site $(\C_G, J_G)$ is a completely connected site that generates the original Grothendieck topos $\E$.
\end{proof}

\section{Examples}\label{sec:examples}

This section aims to provide many examples of completely connected topoi together with their theoretical consequences.
\subsection{Completely connected presheaf topoi}

\begin{proposition}[Completely connected presheaf topoi]\label{prop:Presheafcase}
    For a small category $\C$, its presheaf topos $\PSh(\C)$ is completely connected if and only if the Cauchy completion of $\C$ has an initial object.
\end{proposition}
\begin{proof}
Let $\vC$ denote the Cauchy completion of $\C$.
If $\vC$ has an initial object, $\vC$ equipped with the trivial topology is a completely connected site that generates the topos $\PSh(\vC)$. \Cref{thm:SiteCharacterisation} implies that $\PSh(\C) \simeq \PSh(\vC)$ is a completely connected topos.

Conversely, if the topos $\PSh(\C)$ is completely connected, its container object $\vo$ is connected and projective (\Cref{prop:FundamentalPropertiesOfContainerObject} (1), (2)) and therefore represented by an object of $\vC$.  
Let $I\in \ob(\vC)$ be the representing object of $\vo \in \PSh(\vC) \simeq \PSh(\C)$. For any other object $X\in \ob(\vC)$, the hom-set $\vC(I, X) \cong \PSh(\vo, X)$ is a singleton due to \Cref{prop:FundamentalPropertiesOfContainerObject} (4). This shows that $I$ is initial in $\vC$.
\end{proof}

\begin{remark} The if part of \Cref{prop:Presheafcase} has a more direct proof without mentioning \Cref{thm:SiteCharacterisation}.
    If $\C$ admits an initial object, the unique functor $!\colon \C \to 1$ admits a left adjoint $i\dashv\; !$. This adjunction lifts to the adjunction between their presheaf topoi $i^* \dashv\; !^* = \G_0\colon \Set \to \PSh(\C)$. Then, the functor $i^* = \G_1$ admits a left adjoint given by the left Kan extension along $i$, which witnesses that $\PSh(\C)$ is completely connected.
\end{remark}

We have an immediate corollary.
\begin{corollary}
\label{cor:PresheafDuality}
    For a small category $\C$, $\PSh(\C)$ is completely connected if and only if $\PSh(\C^{\op})$ is local. 
\end{corollary}
\begin{proof}
    One can similarly prove that a presheaf topos $\PSh(\C)$ is local if and only if the Cauchy completion $\vC$ has a terminal object.
\end{proof}

\begin{example}[Sets]\label{exmp:toposOfSetsAsCompletelyCOnnectedTopoi}
    The simplest example of a completely connected topos is the topos of sets $\Set$, which is a presheaf topos over the terminal category $1$. Notice that the degenerate topos $1 \simeq \PSh(\emptyset)$ is not completely connected, since the empty category, which is Cauchy complete, does not have an initial object.
\end{example}

\begin{example}[Sierpi\'nski topos]
\label{exmp:SierpinskitoposAsCompletelyCOnnectedTopoi}
    The Sierpi\'nski topos $\Set^{\to}$ is completely connected since the indexing category $\to$ has an initial object (see also \Cref{exmpl:ToposOfFunctions}). This is a prototypical example of a completely connected topos.
    
    We can regard an object of the Sierpi\'nski topos as a family of sets $\{X_{\lambda}\}_{\lambda\in \Lambda}$. For a given object of $\Set^{\to}$, which is a function $A\to B$, the corresponding family of sets is given by $\{f^{-1}(b)\}_{b\in B}$.
    (This correspondence will be described again in \Cref{exmpl:PresheafCaseOfFamilyConstruction}.)
    From this point of view, a typical object, the container object $\emptyset \to \{\ast\}$, and the initial object $\emptyset \to \emptyset$ look like \Cref{fig:ObjectOfSierpinskitopos}, \Cref{fig:ContainerObjectOfSierpinskitopos}, and \Cref{fig:InitialObjectOfSierpinskitopos}, respectively. This is the reason why we chose our notation $\vo$ for the container object.
\end{example}

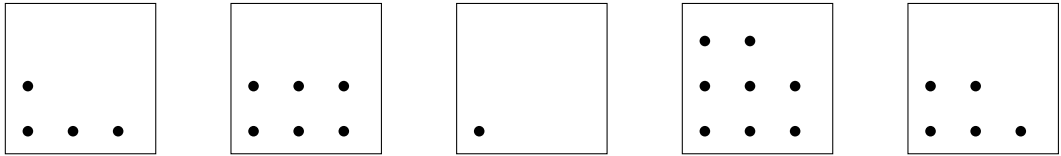
\begin{figure}[ht]
    \centering
    \begin{tikzpicture}
        \def\squareSize{2}
        \foreach \x/\numPoints in {
            0/3,
            3/5,
            6/0,
            9/7,
            12/4
        } {
            \draw (\x,0) rectangle (\x+\squareSize, \squareSize);
            
            \foreach \i in {0, ...,\numPoints} {
                \pgfmathparse{0.3 + mod(\i,3) * 0.6} \let\px\pgfmathresult
                \pgfmathparse{0.3 + floor(\i/3) * 0.6} \let\py\pgfmathresult
                \fill (\x+\px, \py) circle (2pt);
            }
        }
    \end{tikzpicture}
    \caption{An object of the Sierpi\'nski topos}
    \label{fig:ObjectOfSierpinskitopos}
\end{figure}

\begin{figure}[ht]
    \centering
    \begin{tikzpicture}
        \def\squareSize{2}
        \draw (0,0) rectangle (\squareSize, \squareSize);
    \end{tikzpicture}
    \caption{The container object of the Sierpi\'nski topos looks like \dq{an empty box.}}
    \label{fig:ContainerObjectOfSierpinskitopos}
\end{figure}
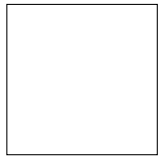
\begin{figure}[ht]
    \centering
    \begin{tikzpicture}
        \def\squareSize{2}
        \draw[white] (0,0) rectangle (\squareSize, \squareSize);
    \end{tikzpicture}
    \caption{The initial object of the Sierpi\'nski topos looks like \dq{nothing.}}
    \label{fig:InitialObjectOfSierpinskitopos}
\end{figure}

\begin{example}[Augmented simplicial sets]\label{exmpl:AugmentedSimplicialSets}
    The topos of augmented simplicail sets is another prototypical example of a completely connected topos. Its container object is the standard $(-1)$-simplex. An augmented simplicial set is regarded as a family of simplicial sets. From this point of view, the container object is the singleton family of an empty simplicial set.
\end{example}

\begin{example}[The object classifier]\label{exmpl:ObjectClassifier}
    The classifying topos of the theory of objects $\Func{\FinSet}{\Set}$ is completely connected. Its container object is the universal object $\iota \colon \FinSet \hookrightarrow \Set$, which is an example of a non-subterminal container object. Similarly, the classifying topos of inhabited objects $\Func{\FinSet_{\neq \emptyset}}{\Set}$ is also completely connected.
\end{example}


\begin{example}[Idempotents]\label{exmpl:ToposOfIdempotentsAsCompletelyConnected}
    The topos of idempotents is completely connected (see also \Cref{exmpl:ToposOfIdempotents}). Its container object is the terninal object $\id_1\colon 1\to 1$, which is an example of a container object $\vo$ that admits a global section $1 \to \vo$. 
    In \Cref{prop:characterisationOfAxiomInfty}, we will prove that a completely connected topos $\E$ satisfies $\Ax{\infty}$ if and only if the container object admits a global section.
\end{example}

\begin{example}[Cosimplicial sets]\label{exmpl:toposOsCosimplicialSets}
    The topos of cosimplicial sets $\PSh(\Delta^{\op})$ is completely connected. This is a typical example of \Cref{cor:PresheafDuality} reflecting the fact that the topos of simplicial sets is local. Similarly, the classifying topos of inhabited objects (\Cref{exmpl:ObjectClassifier}) can be seen as \dq{the dual} of the local topos of \demph{symmetric simplicial sets} $\PSh(\FinSet_{\neq \emptyset})$, which is also called the (non-tirivial) Boolean algebra classifier (see \cite{lawvere1988toposesGenerated}).
\end{example}

\begin{example}[Classifying topoi of equational theories]\label{exmpl:Classyfingtopos}
    Let $\T$ be an equational theory and $\C_\T$ be the category of finitely presented $\T$-algebras. The classifying topos of $\T$-algebras, which is the functor catgeory $\Func{\C_\T}{\Set}$, is completely connected if and only if $\C_\T$ has a terminal object.
    Furthermore, the category $\C_\T$ has a terminal object if and only if the terminal $\T$-algebra is finitely presented (since $\C_\T$ contains the $\T$-algebra freely generated by one element).
    For example, the classifying topoi of groups, rings, abelian groups, monoids, $\R$-vector spaces, and lattices are completely connected. If the number of symbols is much larger than the number of axioms, such as the theory with infinite constant symbols without axioms, the classifying topos is not completely connected. 
    
    Although those classifying topoi always satisfy $\Ax{-2}$, it satisfies $\Ax{\infty}$ if and only if the category of $\T$-algebras has a zero object. For example, the classifying topoi of groups, abelian groups, and monoids satisfy $\Ax{\infty}$, but the classifying topos of rings does not satisfy $\Ax{\infty}$.
\end{example}

\begin{example}[Trees]\label{Exmpl:ToposOfTrees}
We define the category of (rooted) trees $\Trees$ as follows. An object of $\Trees$ is a pair $(G, r)$ of a (directed and possibly infinite) graph $G= (V,E)$ and a vertex $r\in  V$ such that, for every vertex $v\in V$, there exists a unique (directed) path from $v$ to $r$.
A morphism $f\colon (G,r) \to (G',r')$ in $\Trees$ is a graph homomorphism $f\colon G \to G'$ that preserves the root $f(r) = r'$.

The category of trees is a completely connected presheaf topos, because it is equivalent to the presheaf topos over the totally ordered set $(\omega, \leq)$. The equivalence $F\colon \Trees \xrightarrow{\simeq} \PSh(\omega, \leq)$ sends an object $(G= (V,E), r\in V)$ to the presheaf
    $F(G, r) = (\dots \xrightarrow{f_3} X_2 \xrightarrow{f_2} X_1 \xrightarrow{f_1} X_0 ) \in \ob (\PSh(\omega, \leq ))$ 
    where 
    \[
    X_n \coloneqq \{v\in V \mid \text{the distance from $v$ to $r$ is exactly $n+1$}\}\subset V.
    \]
    For $n\geq 1$ and $x\in X_n$, we define $f_n(x) \in X_{n-1}$ as the unique vertex that has an edge $x \to f_n (x)$.

    What is interesting about this topos is that $\Trees$ is naturally equivalent to the category of families of itself, which we will explain in \Cref{exmpl:PresheafCaseOfFamilyConstruction}. In other words, defining the topos of rooted forests $\Forests$ as the category of families of objects of $\Trees$, we have an equivalence of categories $\Trees \simeq \Forests$, which is visualised in \Cref{fig:TreesAndForestsCorrespondence}.
\end{example}

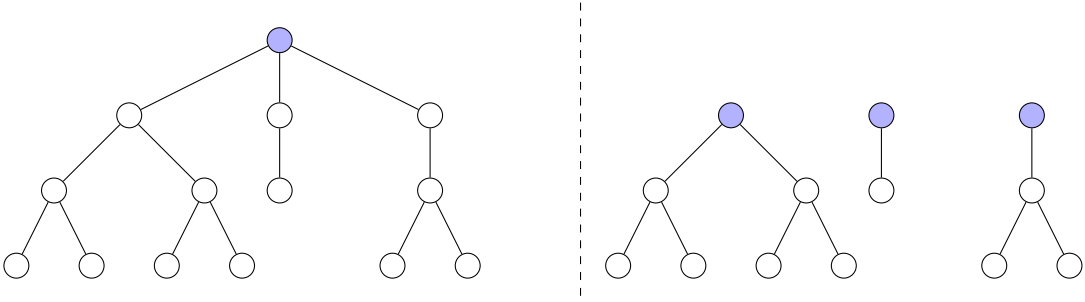
\begin{figure}[ht]
    \centering
    \begin{tikzpicture}
        \node (A) at (0,0) [circle,draw,fill=blue!30] {};
        \node (B) at (-2,-1) [circle,draw] {};
        \node (C) at (0,-1) [circle,draw] {};
        \node (D) at (2,-1) [circle,draw] {};
        \node (E) at (-3,-2) [circle,draw] {};
        \node (F) at (-1,-2) [circle,draw] {};
        \node (G) at (0,-2) [circle,draw] {};
        \node (H) at (2,-2) [circle,draw] {};
        \node (I) at (-3.5,-3) [circle,draw] {};
        \node (J) at (-2.5,-3) [circle,draw] {};
        \node (K) at (-1.5,-3) [circle,draw] {};
        \node (L) at (-0.5,-3) [circle,draw] {};
        \node (M) at (1.5,-3) [circle,draw] {};
        \node (N) at (2.5,-3) [circle,draw] {};

        \draw (A) -- (B);
        \draw (A) -- (C);
        \draw (A) -- (D);
        \draw (B) -- (E);
        \draw (B) -- (F);
        \draw (C) -- (G);
        \draw (D) -- (H);
        \draw (E) -- (I);
        \draw (E) -- (J);
        \draw (F) -- (K);
        \draw (F) -- (L);
        \draw (H) -- (M);
        \draw (H) -- (N);

        \node (B2) at (6,-1) [circle,draw,fill=blue!30] {};
        \node (C2) at (8,-1) [circle,draw,fill=blue!30] {};
        \node (D2) at (10,-1) [circle,draw,fill=blue!30] {};
        \node (E2) at (5,-2) [circle,draw] {};
        \node (F2) at (7,-2) [circle,draw] {};
        \node (G2) at (8,-2) [circle,draw] {};
        \node (H2) at (10,-2) [circle,draw] {};
        \node (I2) at (4.5,-3) [circle,draw] {};
        \node (J2) at (5.5,-3) [circle,draw] {};
        \node (K2) at (6.5,-3) [circle,draw] {};
        \node (L2) at (7.5,-3) [circle,draw] {};
        \node (M2) at (9.5,-3) [circle,draw] {};
        \node (N2) at (10.5,-3) [circle,draw] {};

        \draw (B2) -- (E2);
        \draw (B2) -- (F2);
        \draw (C2) -- (G2);
        \draw (D2) -- (H2);
        \draw (E2) -- (I2);
        \draw (E2) -- (J2);
        \draw (F2) -- (K2);
        \draw (F2) -- (L2);
        \draw (H2) -- (M2);
        \draw (H2) -- (N2);

        \draw[dashed] (4,0.5) -- (4,-3.5);
    \end{tikzpicture}
    \caption{The correspondence between rooted trees and rooted forests}
    \label{fig:TreesAndForestsCorrespondence}
\end{figure}

\begin{example}[Monoid action topos]
Recall that an element $z\in M$ of a monoid $M$ is called a \demph{right zero element} if, for any element $m\in M$, we have $mz=z$. 
\Cref{prop:Presheafcase} implies that, for a monoid $M$, its presheaf topos $\PSh(M)$ is completely connected if and only if $M$ has a right zero element $z\in M$. 
(The proof is as follows: The Cauchy completion of $M$, which is denoted by $\overline{M}$, has idempotent elements of $M$ as objects. Each homset $\overline{M}(i_1,i_2)$ is given by $\{m\in M\mid i_2 m i_1=m\}$.
It is easy to see that a right zero element of $M$, which is idempotent, is an initial object of $\overline{M}$ if it exists. 
Conversely, suppose that an idempotent element $z\in M$ is an initial object in $\overline{M}$. Then, for any element $m\in M$, both $mz$ and $z$ are morphisms from $z$ to the neutral element $e_M\in M$, therefore $mz=z$. This proves that $z$ is a right zero element of $M$.)

Its container object is given by $\vo = \{m\in M \mid zm =m\} $ equipped with the restriction action of the representable $M$-set $M$.
Since the container object $\vo$ is the colimit of the idempotent action $z \ast - \colon M \to M \in \PSh(M)$, it represents the functor sending a $M$-set $X$ to its set of $z$-fixed points:
\[
\{x\in X \mid xz=x\} \cong \PSh(M)(\vo, X) \cong \G_1 (X).
\] 

We can interpret this bijection by regarding the monoid $M$ as a set of operations.
From this point of view, the right zero element $z\in M$ corresponds to the \dq{set-default} operation. Then, the above bijection means that the number of connected components is equal to the number of \dq{default states.}


Let us give an example of this interpretation
inspired by \cite{lawvere1989display}\footnote{Notice that a graphic monoid has a left zero-element (\cite[][Proposition 11]{lawvere1989display}). Therefore its presheaf topos is a local topos, which is dually similar to our case.}.
Consider the operations involved in reading books:
\begin{description}
    \item[$\blacktriangleright$] Move to the next page if there is one.
    \item[$\blacktriangleleft$] Move to the previous page if there is one.
    \item[$z$] Return to the title page.
\end{description}
These operations (with natural equations) generate a monoid $M$, (which can be formally defined as the opposite of the semidirect product $\{\blacktriangleright, \blacktriangleleft\}^* \rtimes \langle z \mid z^2 = z \rangle$,) where $z$ is a right zero element. 
Then, the above bijection means that the number of books is equal to the number of title pages.

\end{example}

\subsection{Completely connected localic topoi}

\begin{lemma}\label{lem:TinyImpliesSubterminalInLocalicTopos}
    For a localic topos $\E$, a projective and connected object is subterminal.
\end{lemma}
\begin{proof}
    Let $X$ be a projective and connected object. Since $\E$ is localic, there is an epimorphism $p\colon \coprod_{\lambda\in \Lambda} U_{\lambda} \twoheadrightarrow X$, where each $U_\lambda$ is subterminal. Since $X$ is projective, the identity morphism $\id_{X}\colon X\to X$ lifts along $p$, which proves that $X$ is a retract of $\coprod_{\lambda\in \Lambda} U_{\lambda}$. The connectedness of $X$ implies that $X$ is a subobject of $U_\lambda$ for a single $\lambda\in \Lambda$. This shows that $X \rightarrowtail U_\lambda \rightarrowtail 1$ is subterminal.
\end{proof}

\begin{proposition}[Completely connected localic topoi]\label{prop:COmpletelyconnectedLocalicTopoi}
    For a locale $L$, its sheaf topos $\Sh(L)$ is completely connected if and only if $L$ has a minimum nonempty open $U_0 \in \O(L)$, i.e., 
    \[
    \forall V\in \O(L), \; (U_0\leq V) \iff (V\neq \bot).
    \]
\end{proposition}
\begin{proof}
    If $\O(L)$ has such $U_0 \in \O(L)$, the canonical site consisting of $\O(L)\setminus \{\bot\}$ is a completely connected site with the initial object $U_0$. \Cref{thm:SiteCharacterisation} implies that $\Sh(L)$ is completely connected.

    Conversely, assume that $\Sh(L)$ is completely connected and let $\vo \in \ob(\Sh(L))$ be its container object. Since a container object is projective and connected (\Cref{prop:FundamentalPropertiesOfContainerObject} (1), (2)), \Cref{lem:TinyImpliesSubterminalInLocalicTopos} implies that $\vo$ is subterminal. Let $U_0$ be the corresponding element of $\O(L)$. For any element $V \in \O(L)$, $U_0 \leq V$ if and only if $V$ is not initial, due to \Cref{prop:FundamentalPropertiesOfContainerObject} (6). This proves that $U_0$ is the minimum nonempty open.
\end{proof}

\begin{corollary}[Completely connected topological space]\label{cor:CompletelyConnectedSpacialTopos}
    For a $T_0$ topological space $X$, its sheaf topos $\Sh(X)$ is completely connected if and only if it has an open dense point $x\in X$.
\end{corollary}
\begin{proof}
A point $x\in X$ is dense if and only if $x\in V \iff V \neq \emptyset$ holds for any open subset $V \subset X$. Therefore, the existence of an open dense point is equivalent to the existence of the minimum open subset of $X$ that is a singleton.

Under the assumption that $X$ is $T_0$, if a minimum nonempty open subset $U_0$ of $X$ exists, $U_0$ is a singleton. Now, \Cref{prop:COmpletelyconnectedLocalicTopoi} completes the proof.
\end{proof}

The topos of trees (\Cref{Exmpl:ToposOfTrees}) is an example of a completely connected localic topos, which is the sheaf topos over the topological space visualised in \Cref{fig:TreesToposTopologicalSpace}.

\begin{figure}
    \centering
    \begin{tikzpicture}
    \def\dx{1.2}
    \def\dy{0.2}
    
    \foreach \x in {0,1,2,3,4,5,6,7,8,9,10} {
        \node[circle,fill,inner sep=1.5pt] (P\x) at (\x * \dx, 0) {};
        \node[below] at (P\x.south) {\x};
    }
    \node[] at (11 * \dx, 0) {$\dots$};
    
    \foreach \n in {0,1,2,3,4,5,6,7,8,9} {
        \draw[dashed] (\n*\dx/2, 0) ellipse ({0.6 + \n * \dx/2} and 0.8+ \n*\dy);
    }
\end{tikzpicture}
    \caption{The underlying topological space of topos of trees, which has an open dense point.}
    \label{fig:TreesToposTopologicalSpace}
\end{figure}
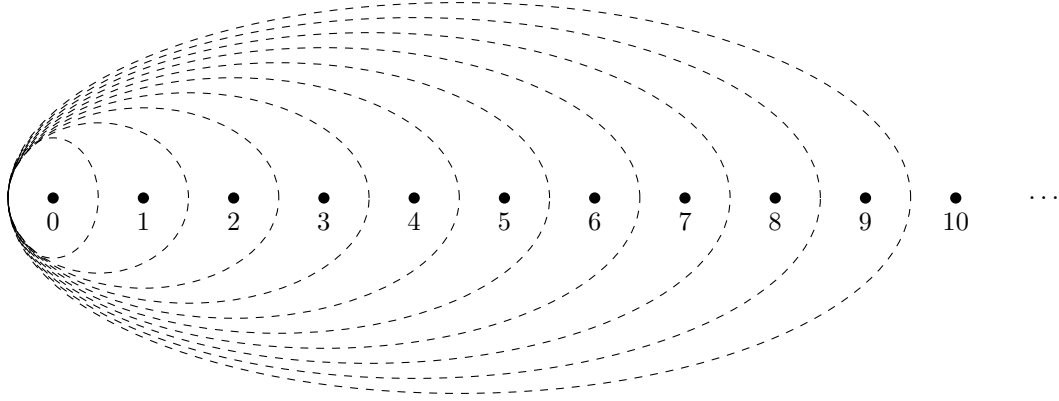

\subsection{Completely connected topoi of families}
\begin{definition}[Category of families]\label{def:FamilyConstruction}
    For a category $\E$, its \demph{category of families} $\Fam(\E)$ has set-indexed families of objects of $\E$ as objects. A morphism from a family $(I, \{X_i\}_{i\in I})$ to another family $(J, \{Y_j\}_{j\in J})$ is a pair $(\alpha, \{f_i\}_{i \in I})$ of a function $\alpha\colon I \to J$ and an $I$-indexed family of morphisms $ \{f_i\colon X_i \to Y_{\alpha(i)}\}_{i \in I}$.
\end{definition}
In other words, $\Fam(\E)$ is the free small coproduct cocompletion of $\E$.
For example, we have $\Fam(\Set) \simeq \Set^{\to}$ and $\Fam(\Set_*) \simeq \PSh(\mathbb{F}_2, 1,\times)$. 

\begin{theorem}[Family construction]\label{thm:FamIsCompletelyConnected}
For any Grothendieck topos $\E$, its category of families $\Fam(\E)$ is a completely connected Grothendieck topos.
\end{theorem}
\begin{proof}
Although there are other proofs that may be shorter, here we prefer a concrete proof, which provides a description of a site for $\Fam(\E)$.
    Let $(\C, J)$ be a small site that generates $\E$. We will construct a completely connected site $(\C^{\tal}, J^\tal)$ such that $\Fam(\E) \simeq \Sh(\C^\tal, J^\tal)$.

    The category $\C^\tal$ is defined as the category $\C$ equipped with a new formal initial object $\emptyset$. The object $\emptyset\in \ob(\C^\tal)$ is a strict initial object in the sense that, for any object $c\in \ob(\C) \subsetneq \ob(\C^\tal)$, there are no morphisms from $c$ to $\emptyset$.

    The topology $J^\tal$ is defined as follows. The only $J^\tal$-covering sieve of the initial object $\emptyset$ is the maximal sieve, that is, the initial object $\emptyset$ is $J^\tal$-irreducible. For other objects $c\in \ob(\C)$, a sieve $S'$ is $J^\tal$-covering if and only if $!\colon \emptyset \to c \in S'$ and $S\coloneqq S'\setminus \{!\colon \emptyset \to c\}$, which is a sieve of $c$ in the original category $\C$, is $J$-covering $S\in J(c)$. Thus, there is a bijection $J(c) \xrightarrow{\cong} J^\tal(c)$ that sends a $J$-covering sieve $S$ to $S\cup \{!\colon \emptyset \to c\}$. It is straightforward to prove that $(\C^\tal , J^{\tal})$ is a site.

    In order to prove $\Fam(\E) \simeq \Sh(\C^\tal, J^\tal)$, we first prove $\Fam(\PSh(\C)) \simeq \PSh(\C^\tal)$. This equivalence is given by sending a presheaf $P \in \ob(\PSh(\C^{\tal}))$ to the $P(\emptyset)$-indexed family $(P(\emptyset), \{P_i|_{\C^\op}\}_{i\in (P(\emptyset)})$, where $P_i \in \PSh(\C^\tal)$ is the subpresheaf of $P$ given by
    \[
    P_i (c) \coloneqq \{x\in P(c) \mid x * (! \colon \emptyset \to c)=i \in P(\emptyset)\}.
    \]
    It is also straightforward to verify that this correspondence gives an equivalence of categories.
    Lastly, one can prove that the $J^\tal$-sheaf condition for a presheaf $P\in \PSh(\C)$ is equivalent to the componentwise $J$-sheaf condition via the equivalence $\Fam(\PSh(\C)) \simeq \PSh(\C^\tal)$. This completes the proof of the equivalence $\Fam(\E) \simeq \Sh(\C^\tal, J^\tal)$.
\end{proof}

\begin{corollary}\label{cor:toposIsClosedSubtoposOfCompletelyConnectedTopos}
    Every Grothendieck topos $\E$ is a closed subtopos of a completely connected Grothendieck topos.
\end{corollary}
\begin{proof}
    Due to \Cref{thm:FamIsCompletelyConnected}, it suffices to prove that $\E$ is a closed subtopos of $\Fam(\E)$. 
    Let $\vo$ be the container object of the completely connected topos $\Fam(\E)$. As a family of objects, $\vo$ is a singleton family of the initial object $\{0_{\E}\}_{*\in \{*\}}$. This is subterminal in $\Fam(\E)$.

    We prove that the closed subtopos of $\Fam(\E)$ corresponding to the subterminal object $\vo \rightarrowtail 1$ is equivalent to $\E$. 
    By the concrete description of closed subtopos (for example, see \cite[][Proposition A.4.5.3]{johnstone2002sketchesv2}) it suffices to prove that $(I, \{X_i\}_{i\in I}) \times \vo \cong \vo$ if and only if $I$ is a singleton. This follows from $\vo = (\{*\}, \{0_{\E}\}_{*\in \{*\}})$ and $(I, \{X_i\}_{i\in I}) \times \vo \cong (I, \{X_i \times 0_{\E}\}_{i\in I}) \cong (I, \{0_{\E}\}_{i\in I})$.
\end{proof}

\begin{remark}[Artin gluing]\label{rmk:FamilyAsArtinGluing}
    The above two results may be understood in terms of \demph{Artin gluing}. (See \cite[][Example A 2.1.12. The gluing construction]{johnstone2002sketchesv1} for more details of Artin gluing.) In fact, for a Grothendieck topos $\E$, its category of families $\Fam(\E)$ is equivalent to the Artin gluing of the finite limit preserving functor $\G_0 \colon \Set \to \E$
    \[
    \Fam(\E) \simeq \Gl(\G_0 \colon \Set \to \E).
    \]
    This proves that $\E$ is a closed subtopos of $\Fam(\E)$, and $\Set \simeq \Sh(1)$ is an open subtopos of $\Fam(\E)$. 
    From the viewpoint of the duality to local topoi, this is dually similar to the Freyd cover construction, which makes a topos $\E$ into a local topos $\overline{\E}$ by taking Artin gluing $\overline{\E} \coloneqq \Gl(\G_{-1}\colon \E \to \Set)$.
\end{remark}

\begin{example}[Presheaf case] \label{exmpl:PresheafCaseOfFamilyConstruction}
For a small category $\C$, we have $\Fam(\PSh(C)) \simeq \PSh(\C^\tal)$. For example, starting from the empty category, which will be denoted by the ordinal number $0$, we obtain the completely connected topos of sets $\PSh(0^\tal) \simeq \Set$ (\Cref{exmp:toposOfSetsAsCompletelyCOnnectedTopoi}). Then, we obtain the Sierpi\'nski topos $\PSh(0^{\tal\tal}) \simeq \Set^{\to}$ (\Cref{exmp:SierpinskitoposAsCompletelyCOnnectedTopoi}). As a \dq{limit step}, we obtain the topos of trees, which is a fixed point of the family construction $\Forests \coloneqq \Fam(\Trees) \simeq \PSh(1+\omega, \leq) \simeq \PSh(\omega, \leq) \simeq  \Trees$ (see \Cref{Exmpl:ToposOfTrees}).
The topos of augmented simplicial sets $\Fam(\PSh(\Delta))$ is also an example of family construction (see \Cref{exmpl:AugmentedSimplicialSets}).
\end{example}

\begin{remark}[The converse direction of \Cref{thm:FamIsCompletelyConnected}]\label{rmk:TheInverseDirection}
    For a completely connected topos $\E$, the following conditions are equivalent:
    \begin{enumerate}
        \item There is a Grothendieck topos $\F$ such that $\E \simeq \Fam(\F)$.
        \item For any object $X \in \E$, $X$ is connected if and only if $ X \times \vo \cong \vo$.
    \end{enumerate}
    We have already seen that $(1)\implies (2)$ in the proof of \Cref{cor:toposIsClosedSubtoposOfCompletelyConnectedTopos}. Conversely, if the topos $\E$ satisfies $(2)$, the connectedness of the container object $\vo$ (\Cref{prop:FundamentalPropertiesOfContainerObject} (1)) implies that $\vo\times \vo \cong \vo$ and therefore $\vo$ is subterminal. Let $\F$ be the closed subtopos corresponding to $\vo \rightarrowtail 1$. The assumption $(2)$ implies that the embedding $\F \hookrightarrow \E$ coincides with the embedding of connected objects $\E_{\conn} \hookrightarrow \E$. This and the local connectedness of $\E$ imply that $\E \simeq \Fam(\F)$.
\end{remark}

Of course, not every completely connected topos is of the form $\Fam(\E)$. For example, the topos of idempotents (\Cref{exmpl:ToposOfIdempotentsAsCompletelyConnected}) is not of this form. In fact, for a Grothendieck topos $\E$, $\E$ cannot satisfy $\Ax{\infty}$ unless $\vo = 1$ (see \Cref{prop:characterisationOfAxiomInfty}).


\subsection{Completely connected topoi with a further left adjoint}

\begin{proposition}\label{prop:characterisationOfAxiomInfty}
    For a Grothendieck topos $\E$, the following conditions are equivalent.
    \begin{enumerate}
        \item $\E$ satisfies $\Ax{\infty}$.
        \item The terminal object of $\E$ is a container object.
        \item $\E$ is completely connected and its container object $\vo$ admits a global section $1 \to \vo$.
    \end{enumerate}
\end{proposition}
\begin{proof}
    If $\E$ satisfies $\Ax{\infty}$, we have $\E(1,-) = \G_{-1} \dashv \G_0$ (\Cref{thm:AxiomsImplication}). This proves that the terminal object $1$ is a container object of $\E$.

    If the terinal object $1\in \ob (\E)$ is a container object of $\E$, then $\E$ is completely connected due to \Cref{prop:ImmediateRephrase}. In this case, it is obvious that the container object $\vo \cong 1$ admits a global section $1\to \vo$.

    Lastly, we assume that $\E$ is completely connected and that its container object $\vo$ admits a global section. Then, the global section $1\to \vo$ is an isomorphism since $\vo$ is rigid (\Cref{prop:FundamentalPropertiesOfContainerObject} (5)). This proves $\G_{-1} = \E(1,-) \cong \E(\vo, -) = \G_1 \dashv \G_0$. Due to \Cref{thm:AxiomsImplication}, $\E$ satisfies $\Ax{\infty}$.
\end{proof}

\begin{example}[Eventually fixed discrete dynamical systems]
    Let $\PSh(\N,0,+)$ be the topos of discrete dynamical systems, which is the presheaf topos over the additive monoid $(\N,0,+)$. An object of $\PSh(\N,0,+)$ is a pair $(X,f)$ of a set $X$ and an endomorphism $f\colon X \to X$. This topos is not completely connected due to \Cref{prop:Presheafcase}.
    Let $\E \subset\PSh(\N,0,+)$ be the full subcategory consisting of discrete dynamical systems $(X,f)$ such that 
    \[
    \forall x\in X,\;  \exists n\in \N,\; f^{n+1}(x) = f^n(x).
    \]
    \invmemo{\cite{hora2024quotient} for more details.}
    This is an example of a Grothendieck topos that satisfies $\Ax{\infty}$ and is not a presheaf topos.
\end{example}

\printbibliography

\end{document}